\newcommand{\C}{\mathbb C }
\documentclass[12pt]{amsart}
\usepackage{amssymb}
\newtheorem{thm}{Theorem}[section]
\newtheorem{cor}[thm]{Corollary}

\theoremstyle{definition}

\theoremstyle{remark}
\newtheorem{rem}[thm]{Remark}
\numberwithin{equation}{section}

\newcommand{\R}{\mathbb R}

\begin{document}

\title[]{Hardy's theorem for compact Lie groups}

\author{S. Thangavelu}

\address{department of mathematics,
Indian Institute of Science, Bangalore - 560 012, India}

\email{veluma@math.iisc.ernet.in}




\begin{abstract} We show that Hardy's uncertainty principle can be reformulated in such a way that it has an analogue even for compact Lie groups and symmetric spaces of compact type.
 
\end{abstract}

\maketitle

 \section{Introduction}

By Hardy's theorem we refer to the following result  proved by Hardy \cite{H}  in 1933 on Fourier transform pairs: if a nontrivial function 
$ f $ and its Fourier transform $ \hat{f} $ on $ \R $ satisfies the conditions
$$ |f(x)| \leq C e^{-ax^2},~~~~~~|\hat{f}(y)| \leq C e^{-b y^2} $$  for $ a, b >0 $ then necessarily $ ab \leq 1/4.$ In other words, if the 
above estimates are valid for any pair  with $ ab > 1/4 $ then $ f $ has to be identically zero. An analogue of this result is true on $ \R^n $ also and 
it is informative to state the result in terms of the heat kernel
$$ p_t(x) = (2\pi t)^{-n/2} e^{-\frac{1}{4t}|x|^2} ,~~~ x \in \R^n $$ associated to the Laplacian $ \Delta $ on $ \R^n.$ If 
$$ \hat{f}(\xi) = (2\pi)^{-n/2} \int_{\R^n} f(x) e^{-i x \cdot \xi} dx $$ then it is well known that $ \hat{p_t}(\xi) = e^{-t|\xi|^2}.$ Hardy's theorem on $ \R^n $ takes the following form: if 
 \begin{equation}
 |f(x)| \leq C p_s(x),~~~~~~|\hat{f}(\xi)| \leq C \hat{p_t}(\xi) 
\end{equation}  with $ 0 < s < t $ then $ f =0;$ when $ s =t , f =C p_t.$\\

The heat kernel version of Hardy's theorem has been extended to Fourier transforms on Lie groups such as non-compact semisimple Lie groups, Heisenberg groups and also for the Helgason Fourier transform on non-compact Riemannian symmetric spaces. In all these cases if $ p_t $ is the heat kernel associated to the suitable Laplacian (or sublaplacian) on the group then estimates of the form (1.1) on $ f $ and its Fourier transform with $ s < t $ always lead to the conclusion 
$ f = 0.$ See \cite{T1} and the references there for the history of Hardy's theorem in various setups. The equality case $ s = t $ remains open in many situations and fails to be true (e.g. $G = SL(2,\R) $) in some cases.\\

But what about Hardy's theorem on compact Lie groups? Well, when $ K $ is a such a group with  $ p_t $ the heat kernel associated to the Laplacian 
$ \Delta $, the condition $ |f(k)| \leq C p_s(k) $ does not say much as any bounded function satisfies the above condition for any $ s > 0 $ since $ K $ is compact and $ p_s(k) > 0.$  Eventhough the condition '$|\hat{f}(\xi)| \leq C \hat{p_t}(\xi) $' has a natural analogue, due to the lack of information from 
$ |f(k)| \leq C p_s(k) $, it looks as though Hardy's theorem has no analogue on compact Lie groups.\\

Nevertheless, let us turn around and have a fresh look at Hardy's theorem- or rather the proof of the theorem on $\R.$ The second condition, namely the
 one on $ \hat{f} $ allows us to extend $ f $ as an entire function to the whole of $ \C $ and we also get the estimate
$$  |f(x+iy)|  \leq C e^{\frac{1}{4b}y^2} ,~~~~~ x, y \in \R.$$ The entire function $ f(z) $ when restricted to $ \R $ satisfies the Gaussian decay 
$ |f(x)| \leq C e^{-ax^2} .$ By considering the equality case $ ab = 1/4 $ Hardy concluded by a clever application of Phragmen-Lindelof theorem that
 $ f(z) = C e^{-az^2} $ and the case $ ab > 1/4 $ follows from this immediately. Thus we see that the growth of the entire function $ f $ and the decay of 
the Fourier transform of its restriction to $ \R $ are related.\\

Viewing Hardy's theorem in the light of the above observation allows us to formulate a version for compact Lie groups.\\

\section{Hardy's theorem for compact Lie groups}
\setcounter{equation}{0}

Let $ K $ be any compact connected Lie group with Lie algebra $ \bf{k}.$ Any inner product on $ \bf{k} $ which is invariant under the adjoint action of $ K $ determines a biinvariant  Riemannian metric on $ K.$ Let $ \Delta$ be the Laplace-Beltrami operator associated  with this Riemannian metric and let $ p_t $ stand for the kernel of the semigroup $ e^{-t\Delta} $ generated by $ \Delta.$ We fix a Haar measure  $ dk $ on $ K.$ Let $ \hat{K} $ stand for the unitary dual of $ K.$ 
Then for any $ \pi \in \hat{K} $ the matrix coefficients $ \pi_{ij}(k) $ of $ \pi $ are eigenfunctions of $\Delta $, see Stein \cite{S}. In particular, the character 
$ \chi_\pi $ is an eigenfunction of $ \Delta.$  Let us denote the corresponding eigenvalues by $ \lambda_\pi^2 $ and let $ d_\pi $ stand for the dimension of the representation $ \pi.$  The heat kernel $ p_t $ has the following expansion:
$$ p_t(k)  = \sum_{\pi \in \hat{K}} d_\pi e^{-\lambda_\pi^2 t} \chi_\pi(k).$$ Note that for $ f \in L^2(K) $ the function $ u(k,t) = f*p_t(k) $ solves the heat equation 
$ \partial_t u(k,t) = -\Delta u(k,t) $ with initial condition $ u(k,0) = f(k).$\\

Let $ G = K_{\C} $ be the universal complexification of $ K $ with Lie algebra $ \bf{k}+i\bf{k}.$ Then $ K $ is a maximal compact subgroup of $ G $ and 
$ X = G/K $ is a non-compact Riemannian symmetric space.  Let $ \Delta_X $ be the Laplace-Beltrami operator on $ G/K $ with associated heat kernel $q_t(g).$ It is a $K-$biinvariant function on $G.$ By the polar decomposition, every element of $ G $ can be written as $ g = k e^{iY} $ where $ Y \in \bf{k}.$ By $ |Y| $ we denote the norm of $ Y $ defined by the $Ad$-$K $ invariant inner product on $ \bf{k}.$ We are now ready state a version of Hardy's theorem for $ K.$\\

\begin{thm} For a  function $ f \in L^2(K) $ let us set  $$ b = \sup \{ t \geq 0:  \|\pi(f)\|_{HS} \leq C e^{- t\lambda_\pi^2}, \forall \pi  \in \hat{K} \}.$$ Suppose 
$ f $ has a holomorphic extension to $ G $ and satisfies the estimate $$ |f(k e^{iY})|^2 \leq C p_{2a}(e^{2iY}) $$ for all $ k \in K $ and $ Y \in \bf{k}.$ Then $ f = 0 $ whenever $ a > b.$

\end{thm}
\begin{proof} Let  $ dg $ is the Haar measure on the complex group $ G.$ We first show that 
\begin{equation}
 \int_{G} |f(g)|^2  q_{t/2}(g) dg < \infty 
\end{equation}
 for every $ 0 < t < a.$  To prove this we make use of several results:  an explicit formula for the heat kernel $ q_t $ due to Gangolli \cite{G}, a theorem of Hall on Segal-Bargmann transforms \cite{H1} and a Gutzmer's formula due to Lassalle \cite{L1}, \cite{L2}.  In order state Gangolli's formula we need to recall some facts about $ G.$  Let $ T $ be a maximal 
torus in $ K $ with Lie algebra $ \bf{t} .$ We identify $ \bf{t}^* $ with $ \bf{k} $ via the inner product on $ \bf{k} $ restricted to $ \bf{t}.$ Let $ R $ be the set of all real roots, $ R^+ $ the set of all positive roots and $ \rho $ be the half sum of positive roots. \\

The polar decomposition of $ G $ reads as $ g = k_1 e^{iH} k_2 $ where 
$ k_1, k_2 \in K $ and $ H \in \bf{t}.$ Then for $ g = k_1 e^{iH} k_2 $ we have 
$$ q_t(g)  = (4\pi t)^{-n/2} e^{-4t|\rho|^2} e^{-\frac{1}{4t} |H|^2} \Pi_{\alpha \in R^+} \frac{(\alpha, H)}{ \sinh(\alpha, H) }.$$
If we let $ \Phi $ stand for the unique $Ad$-$K $ invariant function on $ \bf{k} $ which coicides with $\Pi_{\alpha \in R^+} \frac{(\alpha, H)}{ \sinh(\alpha, H) } $ on $ \bf{t} $ then we can write the formula as
$$ q_t(k e^{iY}) = (4\pi t)^{-n/2} e^{-4t|\rho|^2} e^{-\frac{1}{4t} |Y|^2} \Phi(Y),~~~~ k \in K, Y \in \bf{k}.$$  Since $ 0 < t < a, p_a(g) = p_{a-t}*p_t(g) $ and by the theorem of Hall we have
$$ \int_{G} |p_a(g)|^2 q_{t/2}(g) dg = c \int_{K} |p_{a-t}(k)|^2 dk.$$ 
The expression for the Haar measure on $ G $ in polar cordinates is given by $ dg = c_n \Phi(Y)^{-2} dk dY $ where $ dY $ is the Lebesgue measure on the Lie algebra $ \bf{k}. $  Integrating in polar coordinates, the above leads to 
\begin{equation}
 \int_{K} \int_{\bf{k}} |p_a(ke^{iY})|^2 q_{t/2}(ke^{iY}) \Phi(Y)^{-2} dY dk < \infty.
\end{equation}

Since the heat kernel $ q_t $ is  $ K-$biinvariant, we can rewrite the above as
$$ \int_{K} \int_{\bf{k}} \left(\int_{K\times K} |p_a(uke^{iY}v)|^2 du dv \right) q_{t/2}(ke^{iY}) \Phi(Y)^{-2} dY dk.$$
We can now evaluate the inner integral using Lassalle's formula. Recall that we can write $ Y = Ad(k_1)H, H \in \bf{t} $ for some $ k_1 \in K $ and hence the inner integral above is given by
$$  \int_{K\times K} |p_a(uke^{iY}v)|^2 du dv = \int_{K\times K} |p_a(ue^{iH}v)|^2 du dv $$ which by  Lassalle's formula (Theorem 1 in Section 9  of \cite{L2}) is equal to
$$ \sum_{\pi \in \hat{K}} e^{-2a\lambda_\pi^2}\chi_{\pi}(e^{2iH}) = p_{2a}(e^{2iY}) $$
where we have used the fact that $ p_a $ is a class function. The finiteness of the integral in (2.2) shows that
$$ \int_{K}\int_{\bf{k}} p_{2a}(e^{2iY}) q_{t/2}(k e^{iY}) \Phi(Y)^{-2} dk dY < \infty .$$ The hypothesis on $ f $ now shows that the integral in (2.1) is finite proving our claim. Once again we appeal to the theorem of Hall  to conclude that 
$ f = h*p_t $ for some $ h \in L^2(K).$ But this means that $ \pi(f) = e^{-t\lambda_\pi^2} \pi(h) $ and hence by the definition of $ b$ we conclude that 
$ t \leq b.$ As this is true for any $ t < a $ we get $ a \leq b $ which proves the theorem.

\end{proof}

By using known estimates on the heat kernel $ p_t(g) $ we can restate the above theorem in a more familiar form.

\begin{thm} For a  function $ f \in L^2(K) $ let us set  $$ b = \sup \{ t \geq 0:  \|\pi(f)\|_{HS} \leq C e^{- t\lambda_\pi^2}, \forall \pi  \in \hat{K} \}.$$ Suppose 
$ f $ has a holomorphic extension to $ G $ and satisfies the estimate 
$$ |f(k e^{iY})| \leq C e^{a|Y|^2} $$ for all $ k \in K $ and $ Y \in \bf{k}.$ Then $ f = 0 $ whenever $ a b < 1/4.$

\end{thm}

\begin{cor} Suppose $ F $ is a holomorphic function on $ G = K_{\C} $ which is of exponential type, i.e., it satisfies
$$ |F(ke^{iY})| \leq C e^{a|Y|},~~~~ k \in K, Y \in \bf{k}.$$ Let $ f $ be the restriction of $ F $ to $ K.$ Then 
$ \|\pi(f)\|_{HS} \leq C e^{- t\lambda_\pi^2}, \forall \pi  \in \hat{K} $ for all $ t > 0.$
\end{cor}

\begin{rem} In the equality case of Theorem 2.1,  that is when $ a = b , $ we cannot conclude that $ f = c~  p_a.$ To see this, consider the function $ f $ on $ K= S^1 $ 
whose Fourier expansion is given by
$$ f(x) = \sum_{k=-\infty}^\infty e^{-c|k|-ak^2} e^{ikx}.$$ Then it can be easily checked that the holomorphic extension of $ f $ to $ \C $ satisfies
$$ |f(x+iy)|^2 \leq C p_{2a}(2iy).$$
However, from the proof of the above theorem we see that when $ a=b $ we can write $ f $ as $ h*p_t $ for any $ t < a.$
\end{rem}

\begin{rem} An analogue of Theorem 2.1 is true for any Riemannian symmetric space of compact type. In \cite{St} Stenzel has studied the Segal-Bargmann transform on compact symmetric spaces and proved an analogue of Hall's theorem. If $ p_t $ is the heat kernel assocaited to the Laplace -Beltrami operator on a compact symmetric spaces $ X $ then the image of $ L^2(X) $ under the Segal-Bargmann transform is a weighted Bergman space. The weight function is given by 
$ q_{2t}(e^{2iH}) $ where $ q_t $ is the heat kernel associated to the Laplace-Beltrami operator on the noncompact dual of $ X,$  see Faraut \cite{F} and also \cite{T3}. We leave the formulation and proof of a Hardy's theorem on $ X $ to the interested reader.
\end{rem}

\section{Hardy's theorem on Euclidean spaces revisited}
\setcounter{equation}{0}

For functions on $ \R^n $ the assumption $ |\hat{f}(\xi)| \leq C e^{-b|\xi|^2} $ leads to teh estimate 
$ |f(x+iy)| \leq C p_b(iy)$  where $ p_t $ is the Euclidean heat kernel. However, due to the noncompactness of $ \R^n $ these two estimates cannot be used as in the case of compact Lie groups to 
prove a version of Hardy's theorem. On the other hand, when  $ g $ is a Schwartz function the holomorphic extension of  $ f = g*p_a $ 
satisfies the estimates
$$ |f(x+iy)|^2 \leq C_m (1+|x|^2+|y|^2)^{-m} p_{2a}(2iy) $$ for every non-negative integer $m $ as shown by Bargmann in \cite{B}. This allows us to prove the following  version of Hardy's theorem.\\

\begin{thm}  Let $ b = sup \{ t\geq0: |\hat{f}(\xi)| \leq C e^{-t|\xi|^2} \} $  for a Schwartz class function $ f.$ Assume that $ f $ has a holomorphic extension to $ \C^n $ and satisfies the estimates
 $$  |f(x+iy)|^2  \leq C_m  (1+|x|^2+|y|^2)^{-m} p_{2a}(2iy) $$ for every non-negative integer $m.$ 
 Then for $ f = 0$ whenever $ a > b.$
\end{thm}
\begin{proof} The proof depends on Bargmann's theorem which states that an entire function $ f $ satisfies the hypothesis of the theorem if and only if it is of the form $ g*p_a $ for a Schwartz function $ g $. But then the hypothesis on $ \hat{f} $ forces $ a $ to be at most $ b.$
\end{proof}

The following $ L^2 $ version of the above result can be considered as a reformulation of Cowling-Price theorem (see \cite{T1}).\\

\begin{thm}  Let $ b = sup \{ t\geq 0:  \hat{f}(\xi) = g_t(\xi)e^{-t|\xi|^2}, g_t \in L^2(\R^n) \} $ for a function $ f \in L^2(\R^n).$ Assume that $ f $ has a holomorphic extension to $ \C^n $ and satisfies the estimate
 $$ \int_{\R^n} |f(x+iy)|^2 dx \leq C p_{2a}(2iy).$$ 
Then for $ f = 0 $ whenever $ a > b.$
\end{thm}
\begin{proof} We only need to observe that for any $ t < a , f(x+iy) $ is square integrable with respect to $ p_{2t}(2iy) $ and hence by the theorem of Segal and Bargmann 
we have $ f = h_t*p_t $ for some $ h_t \in L^2(\R^n).$ 
\end{proof}

We also have the following theorem which can be considered as a reformulation of Hardy's theorem for Hermite expansions. By rescaling, the decay rate $ b $ of the Fourier transform of the function $ f $ can be assumed to be at most $ 1/2.$ With this modification we have the following result. We let $ H $ stand for the 
Hermite operator $ -\Delta+|x|^2 $ and write its spectral projection as $ H = \sum_{k=0}^\infty (2k+n) P_k$, see \cite{T2}.\\

\begin{thm} Let $ b = sup \{ t>0: |\hat{f}(\xi)| \leq C e^{-\frac{1}{2}(\tanh 2t)|\xi|^2} \} $ for a  function $ f \in L^2(\R^n).$ Assume that $ f $ has a holomorphic extension to $ \C^n $ and satisfies the estimates
 $$  |f(x+iy)|  \leq C e^{-\frac{1}{2}(\tanh 2a)|x|^2+\frac{1}{2}(\coth 2a)|y|^2} . $$ Then for $ f $ to be nontrivial it is necessary that $ a \leq b.$ Moreover, for every $ t < a $ the Hermite projections of $ f $ have the decay $ \|P_kf\|_2 \leq C e^{-(2k+n)t}.$

\end{thm}
\begin{proof} For the proof we need the following characterisation of the image of $ L^2(\R^n) $ under the Hermite semigroup. If $ f \in L^2(\R^n) $ then 
$e^{-tH}f $ extends to $ \C^n $ as a holomorphic function and we have the identity
$$ \int_{\C^n} |e^{-tH}f(x+iy)|^2 U_t(x,y) dx dy = \int_{\R^n} |f(x)|^2 dx $$ where the weight function $ U_t $ is given explicitly by
$$ U_t(x,y) = c_n (\sinh(2t))^{-n/2} e^{\frac{1}{2}(\tanh 2t)|x|^2-\frac{1}{2}(\coth 2t)|y|^2} .$$ And the converse is also true, see \cite{T2}. Under our assumption 
on $ f $ we see that for any $ t < a $ the holomorphic function $ f(x+iy) $ is square integrable with respect to $ U_t(x,y) dx dy $ and hence $ f(x) = e^{-tH}g(x) $ 
for some $ g \in L^2(\R^n).$ As $ e^{-tH} $ commutes with the Fourier transform we have $ \hat{f} = e^{-tH}\hat{g}.$ In \cite{RT} the authors have shown that  the holomorphic function $ \hat{f}(x+iy) = e^{-tH}\hat{g}(x+iy) $ satisfies the estimate 
$$ |\hat{f}(x+iy)| \leq C e^{-\frac{1}{2}(\tanh 2s)|x|^2+\frac{1}{2}(\coth 2s)|y|^2} $$ for any $ s < t.$ Consequently, we get $ s < b $ and as this is true for any 
$ s<t<a $ we conclude that $ a \leq b $ as claimed.
\end{proof}


 \end{document}